\documentclass[reqno]{amsart}
\usepackage{amssymb}
\usepackage{mathtools}

\newtheorem{theorem}{Theorem}[section]
\newtheorem{lemma}[theorem]{Lemma}

\newtheorem{prop}[theorem]{Proposition}
\newtheorem{claim}[theorem]{Claim}
\theoremstyle{definition}
\newtheorem{definition}[theorem]{Definition}

\theoremstyle{remark}

\newenvironment{claimproof}[1]{\par\noindent\underline{Proof:}\space#1}{\hfill $\blacksquare$}

\usepackage{tikz-cd}

\tikzset{
  symbol/.style={
    draw=none,
    every to/.append style={
      edge node={node [sloped, allow upside down, auto=false]{$#1$}}}
  }
}

\newcommand{\dontprint}[1]\relax

\newcommand{\Sing}{\operatorname{Sing}}

\renewcommand{\P}{{\mathbb P}}
\newcommand{\A}{{\mathbb A}}

\newcommand{\OO}{{\mathcal O}}
\newcommand{\PP}{{\mathcal P}}

\newcommand{\sub}{\subset}

\newcommand{\ov}{\overline}

\newcommand{\lan}{\langle}
\newcommand{\ran}{\rangle}

\newcommand{\rk}{{\operatorname{rk}}}

\renewcommand{\k}{\mathbf{k}}

\usepackage{xcolor}
\usepackage{amscd}

\title{Linear subspaces in cubic hypersurfaces}

\author{Alexander Polishchuk}
\thanks{A.P. is partially supported by the NSF grant DMS-2001224, 
and within the framework of the HSE University Basic Research Program and by the Russian Academic Excellence Project `5-100'.}
\author{Chen Wang}

\address{Department of Mathematics, 
    University of Oregon, 
    Eugene, OR 97403, USA}
\email{}
\address{
    Department of Mathematics, 
    University of Oregon, 
    Eugene, OR 97403, USA; National Research University Higher School of Economics; and Korea Institute for 
    Advanced Study 
  }
  \email{apolish@uoregon.edu}

\begin{document}

\begin{abstract}
We prove that for any cubic polynomial of slice rank $r$, the intersection of all linear subspaces of minimal codimension contained in the corresponding hypersurface
has codimension $\le r^2+\frac{(r+1)^2}{4}+r$ in the affine space. This is deduced from the following result of independent interest. Consider the intersection $I$ of linear ideals $(P_i)$
in $\k[x_1,\ldots,x_n]$, with $\dim P_i\le r$. Then the number of quadratic generators of $I$ is $\le r^2$. 
\end{abstract}

\maketitle

\section{Introduction}

Let $f(x_1,\ldots,x_n)$ be a homogeneous cubic polynomial over a field $\k$. 
Let $\rk(f)$ denote the slice rank of $f$, i.e., the minimal number $r$ such that $f$ belongs to the ideal $(\ell_1,\ldots,\ell_r)$
generated by linear forms. Equivalently, this is the minimal codimension in $\A^n$ of a linear subspace contained in the hypersurface
$f=0$. Let $\PP_f$ denote the set of all subspaces of linear forms $P$ of dimension $\rk(f)$ such that $f\in P$, and let 
$$L_f:=\sum_{P\in \PP_f}P.$$
Geometrically, $L_f$ is the orthogonal to the intersection of all linear subspaces of minimal codimension in the hypersurface $f=0$.
Thus, $\dim L_f$ is the codimension of this intersection in $\A^n$.

In \cite{KP} it was proved that $\dim L_f$ is bounded by a universal polynomial of degree $4$ in $\rk(f)$ (which does not depend on $n$).
Our main result is the following improvement of this bound.

\medskip

\noindent
{\bf Theorem A}. {\it Let $f$ be a cubic of slice rank $r$. Then $\dim L_f\le r^2+\frac{(r+1)^2}{4}+r$.}

\medskip

We deduce this from another result, possibly of independent interest.
For a homogeneous ideal $I\sub S=\k[x_1,\ldots,x_n]$, let us define the number of generators of degree $d$ as $\dim I_d/(S_1\cdot I_{d-1})$.

\medskip

\noindent
{\bf Theorem B}. {\it Let $(P_s)_{s\in S}$ be any collection of subspaces of linear forms, such that $\dim P_s\le r$. Consider the intersection of
the corresponding linear ideals $I=\cap_{s\in S} (P_s)$. Then $I$ has $\le r^2$ generators of degree $2$.}

\medskip

Note that the bound in Theorem A is not far from optimal: we show that the minimal $\dim L_f$ grows quadratically with $r$. 
More precisely, let us set
$$c(r):=\max_{\deg f=3} \dim L_f.$$
Then Theorem A gives a quadratic upper bound on $c(r)$. We find examples of cubics (see Proposition \ref{fn-prop}) that give the following lower bound:
$$c(r)\ge {r+1\choose 2}+r.$$

It would be interesting to determine the exact values of $c(r)$ for small $r$. It is easy to see that $c(1)=3$, and in \cite{KP} it was shown that $c(2)=6$. Here we prove (see Theorem \ref{rk3-thm}) that
$$10\le c(3)\le 12.$$

\section{Intersections of linear ideals}

\subsection{Bound on quadratic generators} We set $S=\k[x_1,\ldots,x_n]$.

\begin{theorem}\label{main-trivial-intersect}
Let $P_1,...,P_s$ be a collection of subspaces of linear forms such that $\dim P_i\le r$ and
 $P_1\cap \cdots \cap P_s=0$.
Consider the ideal $I = (P_1)\cap \cdots \cap (P_s) $. Then $\dim I_2\leq r^2$.
\end{theorem}

\begin{proof}
We show this by induction on the number of spaces $s\ge 2$. In the base case $s=2$ we have $(P_1)\cap (P_2)=(P_1)\cdot (P_2)$, which is generated by $\le r^2$ products
of linear generators of $P_1$ and $P_2$.

Now assume the assertion holds for some $s$. Suppose
we have subspaces of linear forms $P_1,...,P_s,P_{s+1}$ with $P_1\cap \cdots \cap P_{s+1}=0$ and $\dim P_i\le r$. 
Set $P_{1\ldots s}:= P_1\cap \cdots \cap P_{s}$, which is not necessarily trivial, say $P_{1\ldots s} = \langle \ell_1,\ldots,\ell_p\rangle$ with $p = \dim P_{1\ldots s}$. Also,
set $Q:= P_1+\ldots +P_s$, and pick a subspace $C\subset Q$ to be the complement of $P_{1,...,s}$ in $Q$ such that $Q\cap P_{s+1}\subset C$ (we can do this since $P_{1,\ldots,s}\cap P_{s+1}=0$. 

Let $I:=(P_1)\cap\cdots\cap(P_s)$. By \cite[Lemma 3.2]{KP}, we have $I=S I'$, where $I'\sub S(Q)$ is the intersection of the ideals in $S(Q)$ generated by $P_1,\ldots,P_s$.
Consider the following diagram:
$$\begin{tikzcd}[remember picture] 
I'=\pi^{-1}(\ov{I})\arrow[r,symbol=\subset]\arrow[d,"\pi"]& S(Q) \arrow[r,symbol=\supset]\arrow[d,"\pi"] & S(C) \arrow[ld,"\cong"]\\
\ov{I} \arrow[r,symbol=\subset] & S(Q/P_{1\ldots s}) &
\end{tikzcd}$$
where $\pi:S(Q)\to S(Q/(P_{1\ldots s}))$ is the projection, and $\ov{I}=(P/P_{1\ldots s})\cap\ldots \cap (P/P_{1\ldots s})\sub S(Q/P_{1\ldots s})$.
Since $I'=\pi^{-1}(\ov{I})$, it is generated by $\ell_1,\ldots,\ell_p$, together with lifts of generators of $\ov{I}$, which we can choose to be in $S(C)$.
Hence, $I$ has linear generators $\ell_1,\ldots,\ell_p$ and quadratic generators $q_1,\ldots,q_d\in S(C)$. Note that since $\dim P_i/P_{1\ldots s}\le r-p$, by the induction assumption
we have
$$d\le (r-p)^2.$$

Thus, any element in $I_2$ can be written as
$$q=\sum_{i=1}^p m_i\ell_i+\sum_{j=1}^d c_j q_j,$$
where $m_i$ are linear forms and $c_j$ are constants.
Suppose in addition $q\in (P_{s+1})$. Then
$$\sum_{i=1}^p m_i\ell_i\in ((C)\cdot (C)+(P_{s+1}))\cap (P_{1\ldots s}).$$

Let us choose a complement $C'$ to $Q\cap P_{s+1}$ in $C$. Then we have
$$(C)\cdot (C)+(P_{s+1})=(C')\cdot (C')+(P_{s+1}).$$
Since the subspaces $P_{1\ldots s}$, $C'$ and $P_{s+1}$ are linearly independent, we have
$$((C')\cdot (C')+(P_{s+1}))\cap (P_{1\ldots s})=(P_{s+1})\cdot (P_{1\ldots s}).$$
Hence,
$$\sum_i m_i \ell_i \in (P_{1\ldots s})\cdot  (P_{s+1}).$$
It follows that 
$$\dim I_2\cap (P_{s+1})\le \dim ((P_{1\ldots s})\cdot  (P_{s+1}))_2 + d \le p\cdot r+(r-p)^2\le r^2. 
$$
\end{proof}

\medskip

\begin{proof}[Proof of Theorem B]
Set $Q:=\cap_{s\in S} P_s$, $\ov{P}_s:=P_s/Q$. Let us consider the quotient $\ov{S}:=S/(Q)$, which is still a polynomial ring, and let $\pi:S\to \ov{S}$ denote the natural projection.
Since $I=\pi^{-1}(\ov{I})$, where $\ov{I}=\cap_{s\in S}(\ov{P}_s)$, $I$ has the same number of quadratic generators as $\ov{I}$. This reduces us to the case $\cap_{s\in S} P_s=0$.
We can choose a finite subset $S_0\sub S$ such that $\cap_{s\in S_0} P_s=0$. Now, since $I\sub \cap_{s\in S_0} (P_s)$, from Theorem \ref{main-trivial-intersect}, we get $\dim I_2\le r^2$,
as required.
\end{proof}

\subsection{Example}

Here we consider a simple example of intersection of three linear ideals.

\begin{lemma}\label{example-lem}
Consider the ideal $I = (P_1)\cap (P_2)\cap (P_3)$, where $P_1$, $P_2$ and $P_3$ are subspaces of linear forms, such that
$\dim P_i=r$ and $P_i\cap P_j=0$ for $i\neq j$.
Let $k= \dim P_3\cap(P_1+P_2)$. Then 
$$\dim I_2={k \choose 2}.$$
\end{lemma}

\begin{proof}
Without loss of generality we can write
$$      P_1=  \langle x_1, ... , x_r\rangle, \ \ P_2 = \langle y_1,... ,y_r\rangle, \ \ 
    P_3 = \langle x_1+y_1,... , x_k+y_k,z_1, ... ,z_{r-k}\rangle,
    $$
for some linearly independent elements $x_1,\ldots,x_r,y_1,\ldots,y_r.z_1,\ldots,z_{r-k}$.    
The quadratic homogeneous part of $I$ is 
$$ I_2  = \left( (x_1,... ,x_k)\cap (y_1,... ,y_k)\cap(x_1+y_1,... ,  x_k+y_k) \right)_2.$$
It is enough to prove that
\begin{claim}
$I_2 = \langle x_iy_j - x_jy_i | 1\leq i <j\leq k\rangle$.
\end{claim}
\begin{claimproof}
Let $X\sub \P^{2k-1}$ be the projective variety defined by the vanishing locus of all (2,2)-minors of 
$$\begin{pmatrix}
x_1 & x_2 & ... & x_k\\
y_1 & y_2 & ... & y_k
\end{pmatrix}$$
As is well known, $X$ is the image of the Segre embedding
$\mathbb{P}^1\times \mathbb{P}^{k-1} \xhookrightarrow{i} \mathbb{P}^{2k-1}$, and all the quadratic equations of $X$ are linear combinations of $(x_iy_j-x_jy_i)$. 
The projectivizations of our three spaces $(P_1),(P_2),(P_3)$ are the images of $[0:1]\times \mathbb{P}^{k-1}$, $[1:0]\times \mathbb{P}^{k-1}$ and $[1:-1]\times \mathbb{P}^{k-1}$, respectively. 
It now suffices to show that any quadric $q$ vanishing on these three subspaces, vanishes identically on $\mathbb{P}^1\times \mathbb{P}^{k-1}$.

Consider the restriction
$$i^*q \in H^0( i^*\mathcal{O}_{\mathbb{P}^{2k-1}}(2)) = H^0(\mathcal{O}_{\mathbb{P}^1\times \mathbb{P}^{k-1} }(2,2)) = H^0(\mathcal{O_{\mathbb{P}}}(2)) \otimes H^0(\mathcal{O}_{\mathbb{P}^{k-1}}(2)).$$
By assumption, $i^*q$ must vanish on $p_i\times \mathbb{P}^{k-1}$, $i=1,2,3$, where $p_1=[0:1]$, $p_2=[1:0]$, $p_3=[1:-1]$. This implies that
$$ i^*q\in \ker \left( (eval_{p_1}\oplus eval_{p_2}\oplus eval_{p_3}) \otimes id_{H^0(\mathbb{P}^{k-1},\mathcal{O}(2))} \right),$$
where $eval_{p_i}:H^0(\OO_{\P^1}(2))\to \k$ is the evaluation at the point $p_i$. Since $\cap_{i=1}^3\ker(eval_{p_i})=0$, we get that $i^*q=0$.
\end{claimproof}
\end{proof}

\section{Bounds for $c(r)$}

\subsection{Proof of Theorem A}

\begin{definition} Let us say that a polynomial $f\in S=\k[x_1,\ldots,x_n]$ depends on $\le m$ variables if there exists an $m$-dimensional subspace $V$ of linear forms such that
$f\in S(V)\sub S$.
\end{definition}

Note that if $f$ is a homogeneous polynomial, and $f\in S(V)$ for some subspace of linear forms $V$, then any subspace $P\in \PP_f$ is contained in $V$
(since $(P)\cap S(V)\sub (P\cap V)$), hence $L_f\sub V$. Thus, if $f$ depends on $\le m$ variables then $\dim L_f\le m$.

\begin{lemma}\label{qu-cubic-dec-lem}
Let $f$ be a cubic contained in the ideal $I=(P_1)\cap\ldots \cap (P_s)$. Assume that $\cap_{i=1}^s P_i=0$, and set $W=P_1+\ldots+P_s$.
Then $f$ depends on $\le \dim W+\dim I_2$ variables.
\end{lemma}

\begin{proof}
We can write $f$ as
$$f=\sum_{i=1}^{N} \ell_i q_i+\sum_{j=1}^M c_jf_j,$$
where $N=\dim I_2$, $\ell_i$ are linear forms, $c_j$ are constants, and $(q_i)$ (resp., $(f_j)$) are quadratic (resp., cubic) generators of the ideal $I$.
By \cite[Lemma 3.2]{KP}, all these generators can be chosen in $S(W)\sub S$.
Hence, $f$ belongs to $S(W+\lan \ell_1,\ldots,\ell_N\ran)$, so it depends on $\le \dim W+N$ variables.
\end{proof} 

\medskip

\begin{proof}[Proof of Theorem A]
We use induction on $n$. Set $n(r)=r^2+\frac{(r+1)^2}{4}+r$.

Assume first that $\cap_{P\in \PP_f} P=0$. Then we can choose an irredundant collection $P_1,..,P_s\in \PP_f$ with trivial common intersection.
By \cite[Lemma 3.6]{KP}, for $W=P_1+\cdots P_s$ we have 
$$\dim W\leq \frac{(r+1)^2}{4}+r.$$

By Theorem \ref{main-trivial-intersect}, we have
$\dim ((P_1)\cap\cdots\cap (P_s))_2 \leq r^2$. Hence, by Lemma \ref{qu-cubic-dec-lem},
$f$ depends on at most $r^2+\dim W\le r^2+\frac{(r+1)^2}{4}+r=n(r)$ variables, and so $\dim L_f\le n(r)$.

Now assume that $\cap_{P\in \PP_f} P\neq 0$.
Then there is a nonzero $l\in \cap_i P_i $. So we get 
$$\dim L_f=1+\dim L_{\ov{f}},$$ 
where $\ov{f}=f \mod (l)$ is the reduced cubic in the ring $S/(l)\simeq \k[x_1,\ldots,x_{n-1}]$.
Thus, using the induction assumption, we get
\begin{align*}
    \dim L_f&\leq 1+n(r-1) \\
    &=1+(r-1)^2+r-1+\frac{r^2}{4}\\
    &=r+(r-1)^2+\frac{r^2}{4}\\
    &\leq r^2+r+\frac{(r+1)^2}{4} = n(r)
\end{align*}
as required.
\end{proof}

\subsection{Test polynomials $f_n$}

For $n\ge 2$, consider the cubics
$$f_n = \sum_{1\le i<j\le n}x_ix_jy_{ij}$$
in the polynomial rink $\k[x_i, y_{jk} \ | 1\le i\le n, 1\le j<k\le n]$.

\begin{prop}\label{fn-prop}
One has $\rk(f_n)=n-1$, $\dim L_f=n+{n \choose 2}$.
\end{prop}

\begin{proof}
Since for any pair $i<j$ one has $f_n\in (y_{ij},(x_k \ |\ k\neq i,j))$, it suffices to prove $\rk(f_n)\geq n-1$.
We proceed by induction on $n$.
We have $f_2=x_1x_2y_{12}$, which has rank $1$.

Suppose now $n\geq 3$ and $\rk(f_{n-1}) = n-2$, and assume $\rk(f_n)\leq n-2$, i.e. there is an $(n-2)$-dimensional space of linear forms $P=\lan l_1,\ldots,l_{n-2}\ran$ such that $f_n\in (P)$.

First, are going to prove that $P\cap \lan x_1,...,x_n\ran \not=0$. 
By assumption, $f_n = l_1q_1+...+l_{n-2}q_{n-2}$ for some quadratic forms $q_i$. 
Hence, the closed subscheme $V(l_1,\ldots,l_{n-2},q_1,\ldots,q_{n-2})$ is contained in
the singular locus $\Sing(f_n)$ of the hypersurface $f_n=0$. This singular locus is given by the following equations:
\begin{equation}\label{test-poly-singularlocus-n2d3}
    \Sing(f_n) = \left\{
    \begin{aligned}
    &\frac{\partial f_n}{\partial x_i} =\sum_{j\not=i}x_j y_{ij}= 0 \\
    &\frac{\partial f_n}{\partial y_{ij}} = x_i x_j  =  0
  \end{aligned}
  \right.
\end{equation}
where we use the convention $y_{ji}=y_{ij}$.
Clearly, we have $(x_1,...,x_n)\subset \Sing(f_n)$. We claim that this is the irreducible component of the lowest codimension and that all other components have codimension $\ge 2n-2$. 
Indeed, let $Z\sub \Sing(f_n)$ be an irreducible component such that $x_i$ is not identically zero $Z$, for some $i$.
Then by the second condition of \eqref{test-poly-singularlocus-n2d3}, we must have $x_j|_Z =0$ for all $j\neq i$, and by the first condition we get $y_{ij}|_Z=0$ for all $j\neq i$. 
Thus, $Z$ is contained in a linear subspace of codimension $2n-2$, which proves our claim. Thus, we get
\begin{equation*}
   \begin{split}
   \left(\begin{aligned}
    & l_1 = l_2 = \cdots = l_{n-2} = 0\\
        & q_1 = q_2 = \cdots = q_{n-2} = 0
   \end{aligned}\right)
  \end{split}
\subset Sing(f_n) = \mbox{ }
  \begin{split}
    (x_1,...,x_n)\mbox{ } \bigcup \mbox{ }\begin{split}
        \mbox{Some components }\\
        \mbox{ of codim}\geq 2n-2
    \end{split}
  \end{split}
\end{equation*}
Since the locus $V(l_1,\ldots,l_{n-2},q_1,\ldots,q_{n-2})$ has codimension $\leq 2n-4$, it follows that it must be contained in the component $(x_1 = \ldots =x_n = 0)$. 
If $P\cap\lan x_1,\ldots,x_n\ran=0$, then we would get that $V(l_1,\ldots,l_{n-2},q_1,\ldots,q_{n-2})$ is contained in the linear subspace $V(l_1,\ldots,l_{n-2},x_1,\ldots,x_n)$
of codimensions $2n-2$, which is impossible. Hence we deduce that $P\cap\lan x_1,\ldots,x_n\ran\neq 0$.

Reordering the variables, we can assume that $l_1=x_n-c_1x_1-...-c_{n-1}x_{n-1}\in P$. This implies that the restriction of $f_n$ to any linear subspace contained in $l_1=0$ has
slice rank $\le n-3$. But we have
$$ f_n |_{x_n = c_1x_1+...+c_{n-1}x_{n-1}, y_{1n}=...=y_{n-1,n} = 0}  = f_{n-1} $$
so we get a contradiction. Hence, $\rk(f_n)\ge n-1$, as required.
\end{proof}

\section{Cubics of rank $3$}

\subsection{Some lemmas}

\begin{lemma}\label{deg3-trivial-intersect-estimate} Suppose there exist
$P_1,P_2,P_3\in \mathcal{P}_f$ with pairwise trivial intersections. Then $f$ depends on $\le\frac{r(r+3)}{2}$ variables.
\end{lemma}
\begin{proof}
Let $k = \dim P\cap(P_1+P_2)$, then by Lemma \ref{example-lem}, for the ideal $I=(P_1)\cap (P_2)\cap (P_3)$ we have $\dim I_2={k\choose 2}$.
By Lemma \ref{qu-cubic-dec-lem}, $f$ depends on $\le \dim W+\dim I_2$ variables, where $W=P_1+P_2+P_3$ has dimension $3r-k$.
It remains to observe that
$$\dim W+\dim I_2=3r-k + {{k}\choose{2}} = 3r + \frac{k(k-3)}{2}\leq 3r+\frac{r(r-3)}{2} = \frac{r(r+3)}{2}.$$
\end{proof}

\begin{lemma}\label{special-qu-gen-lem}
Let $f$ be a cubic contained in the ideal $I=(P_1)\cap\ldots\cap (P_s)$, where $\cap_i P_i=0$. Set $W=\sum_i P_i$.
Assume we have $I_2\sub (P)$ for some $P\in \PP_f$. Then $P\sub W$.
\end{lemma}

\begin{proof}
Let us write as in the proof of Lemma \ref{qu-cubic-dec-lem},
$$f=\sum_{i=1}^{N} \ell_i q_i+\sum_{j=1}^M c_jf_j,$$
where $(q_i)$ (resp., $(f_j)$) are quadratic (resp., cubic) generators of the ideal $I$.
Note that we have $q_i,f_j\in S(W)\sub S$. Also, by assumption $q_i\in (P)$. It follows that
$q_i\in (P)\cap S(W)\sub (P\cap W)$ and
$$\sum_j c_j f_j\in (P)\cap S(W)\sub (P\cap W).$$
Hence, $f\in (P\cap W)$. Since $\dim P\cap W\le \dim P$, the condition $P\in \PP_f$ implies that $P\cap W=P$, i.e., $P\sub W$.
\end{proof}

\begin{lemma}\label{deg3-commonintersec}
Suppose we have a collection of $3$-dimensional subspaces $(P_i)_{i\in I}$ such that any two intersect in a $2$-dimensional space.
Then either $\dim \cap_{i\in I} P_i \ge 2$, or $\dim \sum_{i\in I}P_i\leq 4$.
\end{lemma}
\begin{proof}
Suppose we could find $P,Q$ in this set such that for a pair $P_1,P_2$, $P_1\cap P_2 \not\subset P\subset P_1+P_2$ and $P_1\cap P_2  \subset Q\not\subset P_1+P_2$. Then $P\cap Q \subset P_1\cap P_2$ and both of them as dimension 2, so they are equal. But this will imply $P\supset P_1\cap P_2$, a contradiction.
So either all spaces intersect in the same 2-dimensional subspace, so any space is contained in $P_1+P_2$, which is 4-dimensional.
\end{proof}

\subsection{Bounds for $c(3)$}

\begin{prop}\label{c3-prop} One has $c(3)\leq 12$. In other words, for any cubic $f$ of slice rank $3$, we have $\dim L_f\le 12$.
\end{prop}

\begin{proof}  
Clearly we can assume that $\PP_f$ contains at least $4$ elements.
We distinguish several cases. 

\medskip

\noindent 
{\bf Case 1}. Assume there exist $P_1,P_2\in \mathcal{P}_f$ such that $ P_1\cap P_2 = 0$, say, $P_1 = \langle x_1,x_2,x_3\rangle, P_2 = \langle x_4,x_5,x_6\rangle $. 
We can also assume that there exists an element $P\in \mathcal{P}_f$ such that
$P\not\subset P_1+P_2$ (otherwise we have $L_f=P_1+P_2$).
There are several possibilities:

\quad 1.a. If $P\cap P_1 = P\cap P_2 = 0$ then by Lemma \ref{deg3-trivial-intersect-estimate}, $f$ depends on $\le 9$ variables, so we are done.

\quad 1.b. If $\dim P\cap P_1 = 1$ and  $\dim P\cap (P_1+P_2) = 1 $, then we can assume that $P = \langle x_1,x_7,x_8\rangle$. In this case the ideal
$I:=(P_1)\cap(P_2)\cap (P)$ has 3 quadratic generators, namely, $x_1x_4,x_1x_5,x_1x_6$. Hence, by Lemma \ref{qu-cubic-dec-lem}, $f$ depends on $\le 8+3=11$ variables. 

\quad 1.c. If  $\dim P\cap P_1 = 1$, $\dim P\cap (P_1+P_2) = 2$, and $P\cap P_2=0$, then we can assume that $P = \langle x_1,x_2+x_4,x_7\rangle$. In this case
$I$ has 3 quadratic generators, hence by Lemma \ref{qu-cubic-dec-lem}, $f$ depends on $\leq 7+3 = 10$ variables.

\quad 1.d. If  $\dim P\cap P_1 = 1$, $\dim P\cap P_2 = 1$, then we can assume that $P=\lan x_1,x_4,x_7\ran$. In this case $I$ has $5$ quadratic generators
$x_1x_4, x_1x_5, x_1x_6, x_2x_4, x_3x_4$, hence by Lemma \ref{qu-cubic-dec-lem}, $f$ depends on $\leq 7+5 = 12$ variables.

\quad 1.e. If $\dim P\cap P_1=2$ and $P\cap P_2 = 0$, then we can assume that $P = \langle x_1,x_2,x_7\rangle$, then there will be 6 quadratic generators in $I$:
$$x_1x_4, x_1x_5, x_1x_6, x_2x_4, x_2x_5, x_2x_6,$$
so Lemma \ref{qu-cubic-dec-lem} would only give that $f$ depends on $\leq 7+6 = 13$ variables.
Let us consider a 4th space $Q\in \mathcal{P}_f$ such that $Q\not\subset P_1+P_2+P$. Using the previous cases with $Q$ instead of $P$,
we see that it is enough to consider one of the two cases: 

\quad 1.e.(i). $ \dim Q\cap P_2=2, Q\cap P_1 = 0$. In this case we can assume $Q = \langle x_4,x_5,x_8\rangle$. Then the intersection $I\cap (Q)$ has $4$ quadratic generators, 
so by Lemma \ref{qu-cubic-dec-lem}, $f$ depends on $\le 8+4 = 12$ variables. 

\quad 1.e.(ii). $ \dim Q\cap P_1=2, Q\cap P_2 = 0$. In this case $Q\cap P_1\cap P$ is either $1$-dimensional or $2$-dimensional. In the former case we can assume
that $Q=\lan x_1,x_3,x_8\ran$, while in the latter case we can assume that $Q=\lan x_1,x_2,x_8\ran$. In the case $Q=\lan x_1,x_3,x_8\ran$, the intersection $I\cap (Q)$ has $3$
quadratic generators, so $f$ depends on $\le 8+3=11$ variables. Finally, in the case $Q=\lan x_1,x_2,x_8\ran$ we note that all quadratic generators of $I$ are contained in 
$\lan x_1,x_2\ran\sub Q$.
Hence, by Lemma \ref{special-qu-gen-lem}, we get $Q\sub P_1+P_2+P_3$, which is a contradiction.


\medskip

\noindent
{\bf Case 2}. Now we assume that any two spaces in $\PP_f$ have nontrivial intersection, and there exist $P_1,P_2\in \mathcal{P}_f$ such that $\dim P_1\cap P_2 = 1$. 
Take another $P\in \PP_f$ such that $P\not\subset P_1+P_2$.
Assume first that $\dim P\cap P_1 = \dim P\cap P_2 =1$, $\dim P\cap (P_1+P_2) = 2$ and $P_1\cap P_2\not\subset P$. Then we can assume that
$P_1 = \langle x_1,x_2,x_3\rangle$, $P_2 = \langle x_1,x_4,x_5\rangle$, $P = \langle x_2,x_4,x_6\rangle$. Hence,
the ideal $(P_1)\cap(P_2)\cap (P_3)$ has 6 quadratic generators, and so $f$ depends $\le 6+6 = 12$ variables.  

If there is no such $P$ then we get that every $P\in \PP_f$ is either contained in $P_1+P_2$ or contains the line $L=P_1\cap P_2$. Let us consider the quotient polynomial ring 
$\ov{S}=S/(L)$ and the reduced cubic $\ov{f}$ which has slice rank $2$. Then we see that $L_f$ is contained in the sum of $P_1+P_2$ with the preimage of $L_{\ov{f}}$.
Note that by \cite[Thm.\ 3.7]{KP}, one has $\dim L_{\ov{f}}\le 6$. Hence, we deduce that $\dim L_f\le \dim(P_1+P_2)+\dim(L_{\ov{f}})\le 5+6=11$.

\medskip

\noindent
{\bf Case 3}. Now assume that any two spaces in $\mathcal{P}_f$ intersect in a 2-dimensional space. Then by Lemma \ref{deg3-commonintersec}, it is enough to consider the case when they have common intersection in a subspace $Q$ of dimension 2. Then the reduced cubic $\ov{f}$ in the quotient $S/(Q)$ has rank $1$, hence, $\dim L_{\ov{f}}\le 3$. This implies that
$\dim L_f\le 2+\dim L_{\ov{f}}\le 5$.
\end{proof}

\begin{theorem}\label{rk3-thm} One has $10\le c(3)\le 12$.
\end{theorem}

\begin{proof}
The inequality $c(3)\le 12$ was proved in Proposition \ref{c3-prop}. On the other hand, by Proposition \ref{fn-prop}, 
the cubic $f_4=\sum_{1\le i<j\le 4} x_ix_j y_{ij}$ has rank $3$ and has $\dim L_f=10$.
\end{proof}

\end{document}